\documentclass{amsart}
\usepackage{amsmath,amssymb}
\usepackage{graphicx,subfigure}
\newtheorem{theorem}{Theorem}[section]

\newtheorem{corollary}[theorem]{Corollary}

\theoremstyle{definition}
 \newtheorem{definition}[theorem]{Definition}

\theoremstyle{remark}
\newtheorem{remark}[theorem]{Remark}

\numberwithin{equation}{section}

\newcommand{\al}{\alpha}
\newcommand{\be}{\beta}
\newcommand{\de}{\delta}
\newcommand{\ep}{\epsilon}

\newcommand{\la}{\lambda}
\newcommand{\om}{\omega}

\newcommand{\De}{\Delta}
\newcommand{\Ga}{\Gamma}

\newcommand{\Si}{\Sigma}
\newcommand{\Om}{\Omega}

\newcommand{\tOm}{\widetilde{\Om}}

\def\RR{\mathbb{R}}

\renewcommand\SS{\mathbb{S}}
\newcommand{\cM}{{\mathcal M}}

\newcommand{\cV}{{\mathcal V}}

\newcommand{\pd}{\partial}
\newcommand\minus\backslash
\newcommand{\id}{{\rm id}}

\newcommand\lan\langle
\newcommand\ran\rangle

\renewcommand\leq\leqslant
\renewcommand\geq\geqslant
\newlength{\intwidth}

\newcommand\loc{_{\mathrm{loc}}}

\newcommand\X{\langle x\rangle}

\begin{document}

\title[Solutions to the Allen--Cahn equation of any
topology]{Bounded solutions to the Allen--Cahn equation\\ with level
  sets of any compact topology}

\author{Alberto Enciso}
\address{Instituto de Ciencias Matem\'aticas, Consejo Superior de
  Investigaciones Cient\'\i ficas, 28049 Madrid, Spain}
\email{aenciso@icmat.es, dperalta@icmat.es}

\author{Daniel Peralta-Salas}

%
%
\begin{abstract}
We make use of the flexibility of infinite-index solutions to the
Allen--Cahn equation to show that, given any compact hypersurface~$\Si$ of~$\RR^d$, with $d\geq 4$, there is a
bounded entire solution of the Allen--Cahn equation on~$\RR^d$ whose
zero level set has a connected component diffeomorphic (and arbitrarily close) to a rescaling
of~$\Si$. More generally, we prove the existence of solutions with a
finite number of compact connected components of prescribed topology
in their zero level sets.
\end{abstract}
\maketitle

\section{Introduction}

The study of the analogies between the level sets of the
solutions to the Allen--Cahn equation
\begin{equation*}
\De u+ u-u^3=0
\end{equation*}
in $\RR^d$ and minimal hypersurfaces in~$\RR^d$ was greatly fostered
by De Giorgi's 1978 conjecture that all the level sets of any entire
solution to the Allen--Cahn equation that is monotone in one direction
have to be hyperplanes for $d\leq8$. This is a natural counterpart
of the Bernstein problem for minimal hypersurfaces, which asserts that
any minimal graph in $\RR^d$ must be a hyperplane provided that
$d\leq8$. Ghoussoub--Ghi and Ambrosio--Cabr\'e proved De Giorgi's
conjecture for $d=2,3$~\cite{GG1,AC}, and the work of
Savin~\cite{Savin} showed that it is also true for $4\leq d\leq8$ under a
weak additional technical assumption. Del Pino, Kowalczyk and Wei~\cite{PKW11}
employed the Bombieri--De Giorgi--Giusti hypersurface to show that the
statement of De Giorgi's conjeture does not hold for $d\geq9$.

In dimension 2, it is well known~\cite{Dancer04} that the monotonicity
hypothesis can be relaxed to the assumption that the solution~$u$ is
{\em stable}\/, i.e., that its Morse index is~0. Let us recall that the {\em
  Morse index}\/ of~$u$ is the maximal dimension of a vector space
$V\subset C^\infty_0(\RR^d)$ such that
\[
\int_{\RR^d} \big( |\nabla v|^2-v^2+3u^2v^2)\, dx<0
\]
for all nonzero $v\in V$. Remarkably, it has been shown recently~\cite{PW} that in
dimension~8 (actually, in any even dimension $d\geq8$) there are
bounded stable solutions to the Allen--Cahn equation whose level sets
are not hyperplanes, but rather they are asymptotic to a minimal
cone. For the role of minimal cones in the Allen--Cahn equation, see
also~\cite{CT} and references therein. In dimensions $d\leq7$, the
level sets of stable solutions to the Allen--Cahn equation are
conjectured to be all hyperplanes~\cite{PW}.

The analysis and possible classification of bounded entire solutions
to the Allen--Cahn equation is an important open problem where the
Morse index of the solutions plays a key role. Unlike the stable case~\cite{Dancer04},
the structure of solutions with finite Morse index can be very
complex; in fact, in dimension~3 a result of del Pino, Kowalczyk and
Wei~\cite{PinoJDG} ensures that, under mild technical assumptions, given any embedded
complete minimal surface in~$\RR^3$ with finite total curvature, there
is a bounded entire solution to the Allen--Cahn equation with a level
set that is close to a large rescaling of this minimal surface, and
that the Morse index of this solution coincides with the genus of the
surface. Also in this direction, the existence of solutions to the
Allen-Cahn equation with a level set close to a nondegenerate minimal
hypersurface was proved by Pacard and Ritor\'e~\cite{PR03} provided
that the ambient space is a compact Riemannian manifold (instead of
$\RR^d$). Furthermore, Agudelo, del Pino and
Wei~\cite{APW} have recently constructed bounded entire axisymmetric solutions
on~$\RR^3$ of arbitrarily large index that have multiple catenoidal ends.

Generally speaking, it is expected~\cite{PNAS,PinoJDG} that the condition that
the Morse index of the solution be finite should play a similar role
as the finite total curvature assumption in the study of minimal
hypersurfaces in Euclidean spaces. In particular, it is well known
that there are many infinite-index solutions to the Allen--Cahn
equation~\cite{Dancer01,CT}, and this abundance of solutions should translate into a
wealth of possible level sets. 

Our objective in this paper is to explore the flexibility of bounded
entire solutions to the Allen--Cahn equation of infinite index by
showing that there are bounded solutions to the Allen--Cahn equation
on~$\RR^d$ with level sets of any compact topology. Specifically,
given a compact hypersurface~$\Si$ without boundary of~$\RR^d$, we
will show that there is a rescaling of $\Si$ that is arbitrarily close
to a connected component of the nodal set of a bounded entire solution
of the Allen--Cahn equation. Furthermore, this level set is {\em
  structurally stable}\/ in the sense that any function on $\RR^d$
which is sufficiently close to~$u$ in the $C^1$~norm in a neighborhood
of this set will also have a zero level set of the same topology. In
view of the existing literature, we are particularly interested in the
case of high dimension~$d$.

To present a precise statement, let us agree to say
that an {\em $\ep$-rescaling}\/ is a diffeomorphism of $\RR^d$ that
can be written as $\Phi=\Phi_1\circ \Phi_2$, where $\Phi_2$ is a
rescaling and $\|\Phi_1-\id\|_{C^1(\RR^d)}<\ep$ (here we could
have taken any other fixed $C^k$ norm, though). By a {\em
  hypersurface}\/ we will refer to a smoothly embedded codimension~1
submanifold of $\RR^d$, so self-intersections will not be allowed. Furthermore, in
what follows we will use the notation $\langle x\rangle:=(1+|x|^2)^{1/2}$ for the Japanese
bracket.

\begin{theorem}\label{T.main}
  Let $\Si$ be any compact orientable hypersurface  without boundary
  of~$\RR^d$, with $d\geq4$, and take any $\ep>0$. Then there is an entire solution~$u$ of the
  Allen--Cahn equation in~$\RR^d$ such that its zero level set
  $u^{-1}(0)$ has a connected component given by $\Phi(\Si)$, where
  $\Phi$ is an $\ep$-rescaling. Furthermore, this set is
  structurally stable and~$u$ falls off at infinity as $|u(x)|<C\langle
  x\rangle^{\frac{1-d}2}$. 
\end{theorem}

It is worth mentioning that the result that we will actually prove
(Theorem~\ref{T.precise}) is in fact stronger, in the sense that given
any finite number of hypersurfaces $\Si_1,\dots,\Si_N$ that are
not linked (see Definition~\ref{D.unlinked}) we will show that there is a
diffeomorphism $\Phi$ such that $\Phi(\Si_1)\cup \cdots\cup
\Phi(\Si_N)$ is a union of connected components of the nodal set of a
bounded entire solution to the Allen--Cahn equation. The
diffeomorphism~$\Phi$ is not an $\ep$-rescaling, although it does act
on each hypersurface~$\Si_j$ as an $\ep$-rescaling composed with
a rigid motion.

The idea of the proof of the theorem is that, when $u$ is small in a
suitable sense, solutions to the Allen--Cahn equation behave as
solutions to the Helmholtz equation
\[
\De w + w=0\,.
\]
Hence a key step of the proof is to establish an analog of
Theorem~\ref{T.main} for solutions to the Helmholtz equations with the
sharp fall-off rate at infinity, which is as $\X^{\frac{1-d}2}$ (Theorem~\ref{T.Helmholtz}). For
this we combine a construction using the first eigenfunction of the
domain bounded by~$\Si$ with a Runge-type theorem with decay
conditions at infinity that generalizes the results that we proved
in~\cite{Annals,Acta} for Beltrami fields on~$\RR^3$. Using suitable weighted
estimates for a convolution operator associated with the Helmholtz
equation (Theorem~\ref{T.G}), we then promote these solutions of the
Helmholtz equation to solutions of the Allen--Cahn equation and show
that the latter still possess a nodal set of the desired
topology. From the method of proof it stems that the statement of
Theorem~\ref{T.main} remains valid for much more general
nonlinearities. (More precisely, one can replace $u^3$ by a smooth enough function
$F(u)$ that behaves as $u^{1+\al}$ as $u\to0$ for some $\al>0$. The
statement then remains valid provided the dimension is larger than
some explicit constant $d_0(\al)$.)

\section{Bounded solutions to the Helmholtz equation}
\label{S.Helmholtz}

In this section we will prove an analog of Theorem~\ref{T.main} for
solutions to the Helmholtz equation on~$\RR^d$. We shall begin by
introducing some notation. 

Let us consider the function
\begin{equation}\label{Gfunct}
G(x):=\be\,|x|^{1-\frac d2}\, Y_{\frac d2-1}(|x|)\,,
\end{equation}
where $Y_{\frac d2-1}$ denotes the Bessel function of the second kind
and we have set
\[
\be:=\frac{2^{1-\frac d2}\pi}{|\SS^{d-1}|\, \Ga(\frac d2-1)}\,,
\]
with $|\SS^{d-1}|$ the area of the unit $(d-1)$-sphere and $\Ga$ the
Gamma function. A simple computation in spherical coordinates shows that
$\De G+G=0$ everywhere but at the origin and the asymptotics for
Bessel functions shows that
\[
G(x)=-\frac1{|\SS^{d-1}|\, |x|^{d-2}}+O(|x|^{3-d})
\]
as $x\to0$. It then follows that $G$ is a fundamental solution for the
Helmholtz equation, so if $v$ is, say, a Schwartz function
on~$\RR^d$ the convolution $G*v$ satisfies
\begin{equation}\label{fundamental}
\De(G*v)+ G*v=v\,.
\end{equation}

As we discussed in the Introduction, we will prove a result that is
considerably more general than Theorem~\ref{T.main}, as it applies to
an arbitrary number of hypersurfaces. There is, however, a topological
condition that we must impose on these hypersurfaces, which is
described in the following
\begin{definition}\label{D.unlinked}
Let $\Si_1,\dots,\Si_N$ be compact orientable
hypersurfaces without boundary of $\RR^d$. We will say that they are {\em
  not linked}\/ if there are $N$ pairwise
disjoint contractible sets $S_1,\dots, S_N$ such that each
hypersurface $\Si_j$ is contained in $S_j$. 
\end{definition}

We are now ready to state and prove the main result of this
section. Notice that the proof of the theorem provides a satisfactory
description of the structure of the diffeomorphism~$\Psi$, as noted in
Remark~\ref{R.rescaling} below. The proof makes use of some techniques we introduced in~\cite{Adv} to study the level sets of harmonic functions and in~\cite{Acta} to construct Beltrami fields with prescribed vortex tubes. Throughout, diffeomorphisms are assumed to be of class
$C^\infty$ and connected with the identity, and $B_R$ denotes the ball
centered at the origin of radius~$R$. Observe that, of course, for
$N=1$ the condition that the hypersurface be not linked is empty,
as it is satisfied trivially.

\begin{theorem}\label{T.Helmholtz}
Let $\Si_1,\dots,\Si_N$ be compact orientable
hypersurfaces without boundary of $\RR^d$ that are not linked, with $d\geq 3$. Then
there is a function~$w$ satisfying the Helmholtz equation
\[
\De w+ w=0
\]
in $\RR^d$ and a diffeomorphism $\Psi$ of $\RR^d$ such that
$\Psi(\Si_1),\dots ,\Psi(\Si_N)$ are structurally stable connected components of
the zero set $w^{-1}(0)$. Furthermore, $w$ falls off at infinity as
$|\pd^\al w(x)|<C_\al\X^{\frac{1-d}2}$ for any multiindex~$\al$.
\end{theorem}

\begin{proof}
An easy application of Whitney's approximation theorem ensures that,
by perturbing the hypersurfaces a little if necessary, we can assume
that $\Si_j$ is a real analytic hypersurface of $\RR^d$. The fact that
the hypersurfaces are not linked allow us now to rescale and translate
them so that the (unique) precompact domains $\Om_j$ that are bounded by each rescaled and
translated real-analytic hypersurface, which we will call $\Si_j':=\pd\Om_j$,
are pairwise disjoint and their first Dirichlet eigenvalue
$\la_1(\Om_j)$ is 1. The first eigenvalue is always simple, so there is a
unique eigenfunction $\psi_j$, modulo a multiplicative constant,
that satisfies the eigenvalue equation
\[
\De \psi_j+\psi_j=0\quad \text{in }
\Om_j\,,\qquad \psi_j|_{\Si_j'}=0\,.
\]
We can choose $\psi_j$ so that it is positive in $\Om_j$.

Hopf's boundary point lemma shows that the gradient of $\psi_j$
does not vanish on~$\Si_j$:
\begin{equation}\label{gradpsij}
\min_{x\in\Si_j}|\nabla\psi_j(x)|>0\,.
\end{equation}
Furthermore, as the hypersurface $\Si_j$ is analytic, it is
standard that $\psi_j$ is analytic in an open
neighborhood $\tOm_j$ of the closure of ${\Om_j}$.

Our goal is to construct a solution $w$ of the
Helmholtz equation in $\RR^d$ that approximates each function
$\psi_j$ in the set $\Om_j$. To this end, let us take a smooth
function $\chi:\RR^d\to\RR$ that is equal to $1$ in a narrow
neighborhood of the closure~$\overline\Om$
and is identically zero outside $\tOm$, with
\[
\tOm:=\bigcup_{j=1}^N\tOm_j\,,\qquad \Om:=\bigcup_{j=1}^N\Om_j\,.
\]
We can now define a smooth function $w_1$ on $\RR^d$ by setting
\[
w_1:=\sum_{j=1}^N\chi\,\psi_j\,.
\]
Here we are assuming that $w_1:=0$ outside $\tOm$.

Since $w_1$ is
compactly supported, we can employ the fundamental solution~\eqref{Gfunct} to write
\begin{equation}\label{intbv}
w_1(x)=\int_{\RR^d} G(x-y)\, f(y)\, dy
\end{equation}
with $f:=\De w_1+w_1$. The support of the function $f$ is obviously
contained in the open set $\tOm\backslash \overline\Om$.
Therefore, an easy continuity argument ensures that one
can approximate the integral~\eqref{intbv} uniformly in the compact
set~$\overline\Om$ by a finite
Riemann sum of the form
\begin{equation}\label{vp1}
w_2(x):=\sum_{n=1}^{M} c_n\, G(x-x_n)\,.
\end{equation}
Specifically, it is standard that for any $\de>0$ there is a large integer~$M$, real numbers ~$c_n$ and points $x_n\in \tOm\backslash \overline\Om$ such that the
finite sum~\eqref{vp1} satisfies
\begin{equation}\label{est1}
\|w_1-w_2\|_{C^0(\Om)}<\de\,.
\end{equation}

Let us now take a large ball $B_R$ containing the closure of the
set~$\tOm$. We shall next show that there is a finite number of points
$\{x_n'\}_{n=1}^{M'}$ in $\RR^d\backslash \overline {B_R}$ and
constants $c_n'$ such that the finite linear combination
\begin{equation}\label{eqmwx2}
w_3(x):=\sum_{n=1}^{M'} c_n'\, G(x-x'_n)
\end{equation}
approximates the function $w_2$ uniformly in $\Om$:
\begin{equation}\label{GtG}
\|w_2-w_3\|_{C^0(\Om)}<\de\,.
\end{equation}
Here $\de$ is the same arbitrarily small constant as above.

Consider the space $\cV$ of all finite linear
combinations of the form~\eqref{eqmwx2} where
$x_n'$ can be any point in $\RR^d\backslash \overline{B_R}$ and the 
constants $c_n'$ take arbitrary values. Restricting these functions to the set $\Om$, $\cV$ can
be regarded as a subspace of the Banach space $C^0(\Om)$ of continuous
functions on $\Om$.

By the Riesz--Markov theorem, the dual of $C^0(\Om)$ is  the space
$\cM(\Om)$ of the finite signed Borel measures on $\RR^d$ whose support
is contained in the set~$\Om$. Let us take any measure
$\mu\in\cM(\Om)$ such that $\int_{\RR^d} f\,d\mu=0$ for all
$f\in \cV$. We now define a function $F\in L^1\loc(\RR^d)$ as
\[
F(x):=\int_{\RR^d} G(x - x')\,d\mu(x')\,, 
\]
so that $F$ satisfies the equation 
\[
\De F+F=\mu\,.
\]
Notice that $F$ is identically zero on $\RR^d\backslash \overline {B_R}$ by the definition of the
measure~$\mu$ and that $F$
satisfies the elliptic
equation
\[
\De F+F=0
\]
in $\RR^d\minus \overline\Om$, so $F$ is analytic in this set. Hence, since
$\RR^d\minus \overline\Om$ is connected and contains the set $\RR^d\minus B_R$, by
analyticity the function $F$ must vanish on the complement of $\Om$. It
then follows that the measure $\mu$ also annihilates the function
$G(\cdot-y)$ with $y\not\in \overline\Om$
because
\[
0=F(y)=\int_{\RR^d} G(y-x')\,d\mu(x')\,.
\]
Therefore 
\[
\int_{\RR^d}w_2\,d\mu=0\,,
\]
which implies that $w_2$ can be
uniformly approximated on~$\Om$ by elements of the subspace $\cV$, due to the Hahn--Banach theorem. Accordingly, there is a finite set of points
$\{x_n'\}_{n=1}^{M'}$ in $\RR^d\backslash \overline {B_R}$ and 
reals $c_n'$ such that the function $w_3$ defined by~\eqref{eqmwx2}
satisfies the estimate~\eqref{GtG}.

To complete the proof of the theorem, notice that the function $w_3$
satisfies the equation
\begin{equation}\label{eqwmi}
\De w_3+w_3=0
\end{equation}
in the ball $B_R$, whose interior contains $\Om$. Let us take hyperspherical
coordinates $r:=|x|$ and $\om:=x/|x|\in\SS^{d-1}$ in $B_R$. Expanding the function
$w_3$ (with respect to the angular variables) in a series of
spherical harmonics and using Eq.~\eqref{eqwmi}, we immediately obtain
that $w_3$ can be written in the ball as a Fourier--Bessel series of the form
\[
w_3=\sum_{l=0}^\infty\sum_{m\in I_l} c_{lm}\, j_l(r)\, Y_{lm}(\om)\,,
\]
where $j_l$ denotes a $d$-dimensional hyperspherical Bessel function,
$Y_{lm}$ are spherical harmonics on $\SS^{d-1}$ and $I_l$ is a finite
set that depends on~$l$ and whose explicit expression will not be
needed here. 

Since the above series converges in $L^2(B_R)$, for any $\de>0$ there is an integer $l_0$
such that the finite sum
\[
w:=\sum_{l=0}^{l_0}\sum_{m\in I_l}  c_{lm}\, j_l(r)\, Y_{lm}(\om)
\]
approximates the function $w_3$ in an $L^2$ sense:
\begin{equation}\label{huwL2}
\|w-w_3\|_{L^2(B_R)}<\de\,.
\end{equation}
By the properties of Bessel functions, $w$ is smooth in $\RR^d$, falls
off as 
\[
|\pd^\al w(x)|\leq C_\al\X^{\frac{1-d}2}
\]
at infinity for any multiindex~$\al$ and satisfies the equation 
\begin{equation}\label{eqmhu}
\De w+w=0
\end{equation}
in the whole space. 

Given any $R'<R$ large enough for the
set $\Om$ to be contained in the ball $B_{R'}$, standard elliptic
estimates allow us to pass from the $L^2$ bound~\eqref{huwL2} to a
uniform estimate
\[
\|w-w_3\|_{C^0(B_{R'})}<C\delta\,.
\]
From this inequality and the bounds~\eqref{est1} and~\eqref{GtG} we infer
\begin{equation}\label{boundmi}
\|w-w_1\|_{C^0(\Om)}<C\de\,.
\end{equation} 
Moreover, since $w_1$ also satisfies the Helmholtz equation in a
neighborhood of
the compact set $\overline\Om$, standard elliptic estimates again imply that
the uniform estimate~\eqref{boundmi} can be promoted to the $C^1$
bound
\begin{equation}\label{lastbound}
\|w-w_1\|_{C^1(\Om)}<C\de\,.
\end{equation}

Finally, since $\Si_1\cup\cdots\cup \Si_N$ is a union of components of the the nodal
set of $w_1$  and the gradient of $w_1$ does not vanish on these
hypersurfaces by~\eqref{gradpsij}, the estimate~\eqref{lastbound} and a direct
application of Thom's isotopy theorem~\cite[Theorem 20.2]{AR} imply
that there is a diffeomorphism $\Psi$ of $\RR^d$ such that
\begin{equation}\label{lvPsi}
\Psi(\Si_1\cup\cdots \cup\Si_N)
\end{equation}
is a union of components of the zero set
$w^{-1}(0)$. Moreover, the diffeomorphism $\Psi$ is $C^1$-close to
the identity. The structural stability of the set~\eqref{lvPsi} for the function
$w$ also follows from Thom's isotopy theorem and the lower bound
\[
\min_{x\in\Psi(\Si_1\cup\cdots \cup\Si_N)}|\nabla w(x)|>0\,,
\]
as a consequence of the
$C^1$~estimate~\eqref{lastbound} and the fact that the function~$w_1$
satisfies the gradient condition~\eqref{gradpsij}.
\end{proof}

\begin{remark}\label{R.rescaling}
It follows from the proof that there are rescalings $\Psi_j^2$,
translations $\Psi_j^3$ and diffeomorphisms $\Psi_j^1$ with
$\|\Psi_j^1-\id\|_{C^1(\RR^d)}$ arbitrarily small such that
\[
\Psi(\Si_j)=(\Psi^1_j\circ\Psi^2_j\circ\Psi^3_j)(\Si_j)\,.
\]
In particular, if $N=1$ the diffeomorphism~$\Psi$ can be assumed to be
an $\ep$-rescaling. A minor modification of the argument would have
allowed us to take $\|\Psi_j^1-\id\|_{C^k(\RR^d)}$ arbitrarily small,
with $k$ any fixed number.
\end{remark}

\section{A weighted estimate for a convolution operator}
\label{S.weight}

In promoting solutions to the Helmholtz equation with sharp decay at
infinity to solutions to the Allen--Cahn equation, the estimates that
we establish in this section will play a key role. 

Specifically, we will be
interested in the convolution of the fundamental solution~$G$,
introduced in Eq.~\eqref{Gfunct}, with
functions with certain decay rate at infinity. To quantify this, for any
nonnegative integer $k$ and any positive real $\nu$ let us denote by
$C^k_\nu(\RR^d)$ the closure of the space of Schwartz functions on
$\RR^d$ with respect to the metric
\[
\|v\|_{k,\nu}:=\max_{|\al|\leq k}\sup_{x\in\RR^d}|\X^\nu\, \pd^\al v(x)|\,.
\]
Clearly 
\[
\| v\, w\|_{k,\nu+\nu'}\leq C\| v\|_{k,\nu}\| w\|_{k,\nu'}
\]
whenever $ v\in C^k_\nu( \RR^d)$ and $ w\in C^k_{\nu'}(\RR^d)$, where $C$ is a constant that only depends on $k$. In
particular,
\begin{equation}\label{Cnus}
\| v^s\|_{k,\nu}\leq C\| v\|_{k, \nu/ s}^s\,.
\end{equation}

The following theorem, which asserts that the convolution with $G$
defines a bounded map $C^k_\nu(\RR^d)\to C^k_{\frac{d-1}2}(\RR^d)$ for
any $\nu>d$, provides the estimates that we need:

\begin{theorem}\label{T.G}
Suppose that $d\geq3$. Then for any $ v\in C^k_\nu(\RR^d)$ with  $k\geq0$ and $\nu>d$, one has
\[
\|G* v\|_{k,\frac{d-1}2}\leq C\| v\|_{k,\nu}
\] 
with a constant that depends on~$d$ and~$\nu$ but not on~$ v$ nor $k$.
\end{theorem}

\begin{proof}
In view of the well known asymptotics for Bessel functions when $d\geq3$, there is a
positive constant $C$ such that $G$ is bounded by
\[
|G(x)|\leq
\begin{cases}
C\, |x|^{2-d} &\text{if } |x|<1\,,\\
C\, |x|^{\frac{1-d}2} &\text{if } |x|>1\,.
\end{cases}
\]
It then follows that
\begin{align}
  |G* v(x)|&\leq \int_{\RR^d} |G(z)|\,| v(x-z)|\, dz\notag\\
            &\leq  C\| v\|_{0,\nu}\,\bigg(\int_{B_1}|z|^{2-d}\, \langle
              x-z\rangle^{-\nu}\, dz + \int_{\RR^d\backslash B_1}|z|^{\frac{1-d}2}\, \langle x-z\rangle^{-\nu}\, dz\bigg)\,.\label{Gvp}
\end{align}
For any fixed $x$, the first integral is convergent for any value of $\nu$, while the second
converges provided that $\nu>\frac{d+1}2$. Since $\nu>d>\frac{d+1}2$, we infer that $G* v(x)$ is well
defined as a convergent integral for any $ v\in C^0_\nu(\RR^d)$ and
all $x\in\RR^d$, and
it only remains to analyze its behavior for large $|x|$.

For concreteness, let us assume that $|x|>2$. We shall next show that
the integrals
\begin{align*}
I_1&:=\int_{B_1}|z|^{2-d}\, \langle
              x-z\rangle^{-\nu}\, dz\\[1mm]
I_2&:=\int_{B_{|x|/2}}|z|^{\frac{1-d}2}\, \langle
     x-z\rangle^{-\nu}\, dz\,,\\[1mm]
I_3&:=\int_{B_{2|x|}\backslash B_{|x|/2}}|z|^{\frac{1-d}2}\, \langle x-z\rangle^{-\nu}\, dz\,,\\[1mm]
I_4&:=\int_{\RR^d\backslash B_{2|x|}}|z|^{\frac{1-d}2}\, \langle x-z\rangle^{-\nu}\, dz
\end{align*}
are then bounded as
\begin{equation}\label{estIj}
I_j< C|x|^{\frac{1-d}2}\,,
\end{equation}
where $C$ does not depend on~$ v$. In view of the inequality~\eqref{Gvp} and the fact that
\[
\int_{\RR^d\backslash B_1}|z|^{\frac{1-d}2}\, \langle
x-z\rangle^{-\nu}\leq I_2+I_3+I_4\,,
\]
this shows that the convolution with $G$ is a bounded map
$C^0_\nu(\RR^d)\to C^0_{\frac{d-1}2}(\RR^d)$. Since for any
multiindex~$\al$ we have
\[
\pd^\al(G* v)=G*(\pd^\al v)\,,
\]
this immediately implies that the convolution with $G$ is also a
bounded map $C^k_\nu(\RR^d)\to C^k_{\frac{d-1}2}(\RR^d)$, thereby
proving the theorem.

So it only remains to prove the estimate~\eqref{estIj} for
$1\leq j\leq4$. For this we will start by using the elementary inequality 
\[
\langle x-z\rangle > 
\begin{cases}
\frac12|x| &\text{if }|z|<\frac{|x|}2\,,\\[1mm]
\frac12|z| &\text{if }|z|>2|x|
\end{cases}
\]
to obtain, for $|x|>2$,
\begin{align*}
I_1&< C|x|^{-\nu}\int_{B_1}|z|^{2-d}\, dz=C|x|^{-\nu}<C|x|^{-d}< C|x|^{\frac{1-d}2}\,,\\[1mm]
I_2&< C|x|^{-\nu}\int_{B_{|x|/2}}|z|^{\frac{1-d}2}\, dz= C|x|^{\frac{d+1}2-\nu}< C|x|^{\frac{1-d}2}\,,\\[1mm]
I_4&< C\int_{\RR^d\backslash B_{2|x|}}|z|^{\frac{1-d}2-\nu}\, dz=C|x|^{\frac{d+1}2-\nu}<C|x|^{\frac{1-d}2}\,.
\end{align*}
To obtain these bounds we have used that $\nu>d$ by assumption.

To estimate $I_3$ we choose a Cartesian basis so that $x=|x|\, e_1$
and then use the rescaled variable $\bar z:=z/|x|$ to write
\begin{align}
I_3&= \int_{B_{2|x|}\backslash B_{|x|/2}}|z|^{\frac{1-d}2}\langle x-z\rangle^{-\nu}\,
     dz\notag\\
&= |x|^{\frac{d+1}2-\nu}\int_{B_{2}\backslash B_{1/2}}\frac{|\bar
  z|^{\frac{1-d}2}\, d\bar
  z}{(\frac1{|x|^2}+|e_1-\bar z|^2)^{\frac\nu2}} \,.\label{I3}
\end{align}
Denoting by $B'$ the ball centered at $e_1$ of radius $\frac14$, one
can check that
\begin{align*}
\int_{B'} \frac{|\bar
  z|^{\frac{1-d}2}\, d\bar
  z}{(\frac1{|x|^2}+|e_1-\bar z|^2)^{\frac\nu2}}  &<C\int_0^{1/4}\frac{\rho^{d-1}\, d\rho}{(\frac1{|x|^2}+\rho^2)^{\frac\nu2}}\\
&=C|x|^{\nu-d}\int_0^{|x|/4}\frac{\bar\rho^{d-1}\,
  d\bar\rho}{(1+\bar\rho^2)^{\frac\nu2}}\\
&<C|x|^{\nu-d}\int_0^{\infty}\frac{\bar\rho^{d-1}\, d\bar\rho}{(1+\bar\rho^2)^{\frac\nu2}}\\
&<C|x|^{\nu-d}
\end{align*}
where we have defined $\bar\rho:=|x|\rho$, and we have used that the integral in $\bar\rho$ is convergent for any $\nu>d$. Plugging this in~\eqref{I3}, one gets that
\begin{align*}
I_3&= |x|^{\frac{d+1}2-\nu}\bigg(\int_{B'}\frac{|\bar
  z|^{\frac{1-d}2}\, d\bar
  z}{(\frac1{|x|^2}+|e_1-\bar z|^2)^{\frac{\nu}{2}}} + \int_{B_{2}\backslash
     (B_{1/2}\cup B')}\frac{|\bar
  z|^{\frac{1-d}2}\, d\bar
  z}{(\frac1{|x|^2}+|e_1-\bar z|^2)^{\frac{\nu}{2}}} \bigg)\\
&<C|x|^{\frac{d+1}2-\nu}(|x|^{\nu-d}+C)\\
&<C|x|^{\frac{1-d}2}\,.
\end{align*}
To obtain the first inequality we have used that $|e_1-\bar
z|^2\geq\frac{1}{16}$ for all $\bar z\in B_{2}\backslash
     (B_{1/2}\cup B')$, and the second inequality follows from the assumption $\nu>d$.  This is the last estimate that we needed in~\eqref{estIj} and thus the theorem follows.
\end{proof}

In view of the structure of the nonlinearity of the Allen--Cahn
equation, the following corollary will be useful:

\begin{corollary}\label{C.G}
For any $v\in C^k_{\frac{d-1}2}(\RR^d)$ with  $k\geq0$ and $d\geq 4$, one has the estimate
\[
\|G*( v^3)\|_{k,\frac{d-1}2}\leq C\| v\|_{k,\frac{d-1}2}^3\,.
\]
\end{corollary}
\begin{proof}
We can apply Theorem~\ref{T.G} with $\nu:=\frac{3d-3}2$ because $\nu>d$ for
all dimensions $d\geq4$, thus implying that
\begin{align*}
  \|G*( v^3)\|_{k,\frac{d-1}2}\leq
  C\| v^3\|_{k,\frac{3d-3}2}\leq C\| v\|_{k,\frac{d-1}2}^3\,,
\end{align*}
where we have used the relation~\eqref{Cnus}.
\end{proof}

\section{Proof of Theorem~\ref{T.main}}
\label{S.main}

We are now ready to prove the main result of this paper, which reduces to
Theorem~\ref{T.main} when $N=1$. 

\begin{theorem}\label{T.precise}
  Let $\Si_1,\dots, \Si_N$ be compact orientable hypersurfaces without
  boundary of~$\RR^d$ that are not linked, with $d\geq4$, and let us
  take any positive integer~$k$. Then there is a
  diffeomorphism $\Phi$ of~$\RR^d$ such that $\Phi(\Si_1),\dots , \Phi(\Si_N)$ are
  connected components of the level set $u^{-1}(0)$ of a smooth solution to
  the Allen--Cahn equation in $\RR^d$ that is bounded as
  $|\pd^\al u(x)|<C_\al\X^{\frac{1-d}2}$ for any multiindex with~$|\al|<k$. Furthermore, these level
  sets are structurally stable and the diffeomorphism~$\Phi$ can be
  assumed to have the same structure as in Remark~\ref{R.rescaling}.
\end{theorem}

\begin{proof}
By Theorem~\ref{T.Helmholtz} there is a solution~$w$ to the Helmholtz
equation
\[
\De w +w=0
\]
on $\RR^d$ such that $\Psi(\Si_1),\dots,\Psi(\Si_N)$ are connected
components of its zero set $w^{-1}(0)$, where $\Psi$ is a
diffeomorphism of~$\RR^d$. Moreover, $\|w\|_{k,\frac{d-1}2}<C$ and the
above hypersurfaces are structurally stable in the sense that there exist
a large ball $B_R$ and a positive constant~$\eta$ such that, if $w'$
is any function with
\begin{equation}\label{ww'}
\|w-w'\|_{C^1(B_R)}<\eta\,,
\end{equation}
then there is a diffeomorphism~$\Phi$ of~$\RR^d$ such that
\begin{align}\label{deformedSis}
\Phi(\Si_1)\cup\cdots\cup\Phi(\Si_N)
\end{align}
are structurally stable connected components of the level set
$w'^{-1}(0)$. Furthermore, $\Phi$ is close to $\Psi$ in the norm~$C^1(\RR^d)$.

Let us take a small positive constant~$\ep$ that will be fixed later
and consider the iterative scheme
\begin{align}
u_0&:=\de w\,,\notag\\
u_{n+1}&:=\de w + G*(u_n^3)\,,\label{iter2}
\end{align}
where we have set
\[
\de:=\frac\ep{2\|w\|_{k,\frac{d-1}2}}\,.
\]
Our goal is to show that if $\ep$ is small enough, $u_n$ converges in
$C^k_{\frac{d-1}2}(\RR^d)$ to a function~$u$ that satisfies the
Allen--Cahn equation
\[
\De u + u-u^3=0
\]
and is close to $\de w$ in a suitable norm.

A first observation is that, if $\|u_n\|_{k,\frac{d-1}2}<\ep$ and
$\ep$ is small enough, by the definition of~$\de$ we
automatically have
\begin{align}
  \|u_{n+1}\|_{k,\frac{d-1}2}&\leq
  \de\|w\|_{k,\frac{d-1}2}+\|G*(u_n^3)\|_{k,\frac{d-1}2}\notag\\
  &\leq \de\|w\|_{k,\frac{d-1}2}+C\|u_n\|_{k,\frac{d-1}2}^3\notag\\ &\leq
  \frac\ep2+C\ep^3\notag\\
&<\ep\,.\label{estu1}
\end{align}
Here we have used Corollary~\ref{C.G} to estimate $G*(u_n^3)$. Notice
that the smallness that we have to impose on~$\ep$ only depends on the
constant that appears in Corollary~\ref{C.G}. In particular, since the first function $u_0$ of the iteration satisfies
$$
\|u_0\|_{k,\frac{d-1}2}=\frac{\ep}2\,,
$$
the induction property~\eqref{estu1} then implies that 
\begin{equation}\label{bun}
\|u_n\|_{k,\frac{d-1}2}<\ep
\end{equation}
for all~$n$.

To estimate the difference $u_{n+1}-u_n$, let us start by noticing
that for any functions $v,v'$ we have
\begin{align*}
\|v^3-v'^3\|_{k,\frac{3d-3}2}&=\max_{|\al|\leq k}\sup_{x\in\RR^d}\X^{\frac{3d-3}2}\,\big|\pd^\al \big(v^2(v-v')+vv'(v-v')+v'^2(v-v')\big)\big|\\
&\leq C\,(\|v\|^2_{k,\frac{d-1}2}+\|v'\|^2_{k,\frac{d-1}2})\,\|v-v'\|_{k,\frac{d-1}2}\,.
\end{align*}
It then follows from Theorem~\ref{T.G}, the fact
that~$\frac{3d-3}2>d$ when $d\geq4$, and Eq.~\eqref{bun} that we can write
\begin{align}
\|u_{n+1}-u_n\|_{k,\frac{d-1}2}&=\|G*(u_n^3-u_{n-1}^3)\|_{k,\frac{d-1}2}\notag\\
&\leq C\|u_n^3-u_{n-1}^3\|_{k,\frac{3d-3}2}\notag\\
&\leq C\ep^2\, \|u_n-u_{n-1}\|_{k,\frac{d-1}2}\,.\label{estu2}
\end{align}
If $\ep$ is small enough for $C\ep^2<\frac12$, it is standard
that~\eqref{estu1} and~\eqref{estu2} imply that~$u_n$ converges in
$C^k_{\frac{d-1}2}(\RR^d)$ to some function~$u$ with
\begin{equation}\label{uep3}
\|u\|_{k,\frac{d-1}2}\leq\ep\,.
\end{equation}
Since the map $v\mapsto G*(v^3)$ is continuous in
$C^k_{\frac{d-1}2}(\RR^d)$, from~\eqref{iter2} we infer that $u$
satisfies the integral equation
\begin{equation}\label{Gu}
u=\de w+G*(u^3)\,.
\end{equation}
As $w$ is a solution of the Helmholtz equation and $G$ is a
fundamental solution satisfying~\eqref{fundamental}, it then follows that
\[
\De u + u= u^3\,,
\]
so $u$ is a solution of the Allen--Cahn equation, which is smooth by elliptic regularity.

One can now use the bound~\eqref{uep3}, the relation~\eqref{Gu} and
the definition of~$\de$ to write
\begin{align*}
\bigg\|w-\frac u\de\bigg\|_{k,\frac{d-1}2}=\frac1\de\, \|\de
                                            w-u\|_{k,\frac{d-1}2}=\frac1\de\,
                                            \|G*(u^3)\|_{k,\frac{d-1}2}\leq
                                            \frac
                                            C\de\|u\|^3_{k,\frac{d-1}2}\leq C\ep^2\,.
\end{align*}
In view of the stability estimate~\eqref{ww'}, if $\ep$ is small
enough (namely, $C\ep^2<\eta$), we infer that there is a
diffeomorphism~$\Phi$, close to the diffeomorphism~$\Psi$ in the norm~$C^1(\RR^d)$, such that the hypersurfaces~\eqref{deformedSis}
are structurally stable connected components of the level set $u^{-1}(0)$. The theorem
then follows.
\end{proof}

\section*{Acknowledgments}

The authors are supported by the ERC Starting Grants~633152 (A.E.) and~335079
(D.P.-S.). This work is supported in part by the
ICMAT--Severo Ochoa grant
SEV-2011-0087.

\bibliographystyle{amsplain}

\end{document}